\documentclass[12pt]{article}
\usepackage{amsmath,amsthm}
\usepackage{amssymb}
\usepackage{amsfonts}

\newtheorem{thm}{Theorem}[section]
\newtheorem{pro}[thm]{Proposition}
\newtheorem{cor}[thm]{Corollary}

\theoremstyle{definition}
\newtheorem{dfn}[thm]{Definition}
\newtheorem{fct}[thm]{Fact}

\newtheorem{rmk}[thm]{Remark}

\newtheorem{que}[thm]{Question}

=22cm\headheight=0cm\topskip=0cm\topmargin=0cm

\begin{document}
\def\dis{\displaystyle}


\def\dotminus{\mathbin{\ooalign{\hss\raise1ex\hbox{.}\hss\cr
  \mathsurround=0pt$-$}}}

\begin{center}
{\LARGE{  Continuous logic and the strict order property
 }} \vspace{10mm}

{
{\bf Karim Khanaki }} \vspace{3mm}

{\footnotesize
   Department of science,\\
  Arak University of Technology,\\
P.O. Box 38135-1177, Arak, Iran; \\
 e-mail: khanaki@arakut.ac.ir

}\vspace{5mm}
\end{center}

{\sc Abstract.}
{\small  We generalize a theory of Shelah for continuous logic,
namely a continuous theory has OP if and only if it has IP or SOP.

\medskip






\section{SOP in Continuous Logic}
We assume that the reader is familiar with continuous logic from
\cite{BBHU} and \cite{BU}. We introduce a notion of `strict order
property' for continuous logic as a complimentary to NIP: a
theory has OP iff it has IP or SOP. We note that the usual
translation of SOP in classical logic to continuous logic is not
the `suitable' notion, because it seems that Shelah's theorem
does not hold with this translation. So, we need to provide a
different definition.

\begin{dfn}
(i) We say a formula $\phi(x,y)$ in continuous logic has the {\em
  strict order property} (SOP) if there exists a sequence
$(a_ib_i:i<\omega)$ in the monster model $\mathcal U$ and
$\epsilon>0$ such that for all $i<j$,
$$\phi({\mathcal U},a_i)\leqslant\phi({\mathcal U},a_{i+1})\ \ \ \mbox{ and } \ \ \phi(b_j,a_i)+\epsilon<\phi(b_i,a_j).$$
We say that the theory $T$ has SOP if a formula $\phi(x,y)$ has
SOP.

 (ii) We say a theory $T$ has the {\em
  weak strict order property} (wSOP) if there are a formula
  $\phi(x,y)$ and $\epsilon>0$ such that for each natural number $n$ there are a formula $\psi_n(x,y)$ (of combination of instances $\phi(x,a)$) and an indiscernible sequence
$(a_i)_{i<\omega}$ and arbitrary sequence $(b_i)_{i<\omega}$ such
that for all $b\in{\mathcal U}$, the sequence $\psi_n(b,a_i)$ has
an eventual value and for all $i<j$,
$$\psi_n({\mathcal U},a_i)\dotminus \psi_n({\mathcal U},a_{i+1})\leqslant\frac{\epsilon}{n}\ \ \ \mbox{ and } \ \ \psi_n(b_j,a_i)+\epsilon<\psi_n(b_i,a_j).$$
In this case, we say the formula $\phi(x,y)$ {\em
 makes the weak strict order property} (or makes wSOP).

 (iii) We say a formula $\phi(x,y)$  has \textbf{not} the {\em
  weak sequential completeness property} (NSCP) if there exists an indiscernible sequence
$(a_i)_{i<\omega}$, an arbitrary sequence $(b_i)_{i<\omega}$ and
$\epsilon>0$ such that for all $b\in{\mathcal U}$, the sequence
$\phi(b,a_i)$ has an eventual value and for all $i<j$,
$\phi(b_j,a_i)+\epsilon<\phi(b_i,a_j)$. We say that a theory $T$
has NSCP if a formula $\phi(x,y)$ has NSCP.

The acronym SOP (wSOP) stands for the (weak) strict order property
and NSOP (NwSOP) is its negation. The acronym SCP stands for the
negation of NSCP.
 \end{dfn}

\begin{que}
Is wSOP (or NSCP) an `expressible' property? In the above
definition, the notion `combination of instances $\phi(x,a)$' is
not expressible.
\end{que}


\begin{rmk} (i) Clearly SOP implies wSOP. (Indeed, let $\psi_n=\phi$ for all $n$.)
 Also, wSOP implies NSCP. (Indeed, let $\phi=\psi_1$.) We will shortly show that SCP and NwSOP are the same. Of course, in classical ($\{0,1\}$-valued) logic, NSCP, wSOP and  SOP are the same.

(ii) We will see shortly that OP implies IP or wSOP, but we could
not prove that OP implies IP or SOP. The reason for this is that
the usual argument of the proof of Shelah's theorem does no hold
for non-discrete-valued logics. So we believe that the correct
notion of strict order property for continuous logic is wSOP.

 (iii) We note that every formula of the form
$\psi(y_1,y_2)=\sup_x(\phi(x,y_1) \dotminus \phi(x,y_2))$ defines
a continuous pre-ordering (see Question~4.14 of \cite{Ben2} for
the definition), in analogy with formulae of the form
$\psi(y_1,y_2)=\forall x(\phi(x,y_1)\to\phi(x,y_2))$ in classical
logic. It is easy to see that for a theory $T$ (in continuous
logic), some formula has SOP if and only if
 there is a formula in $T$ defining a pre-order (in the sense of
 \cite{Ben2})
with infinite chains.

(iv) In the definitions of NSCP and NwSOP we supposed that the
sequence are eventually constant. The reason for this is that we
want the sequence $\phi(x,a_i):S_\phi(\mathcal U)\to \{0,1\}$
converges. In the definition of SOP, since the sequence
$\phi(\mathcal U,a_i)$ is increasing, this requirement is
guaranteed.

(v) Note that contrary to SOP, the property NSCP is not an
`expressible' property of formulas. In fact this property is from
functional analysis:  a Banach space $X$ is called {\em weakly
sequentially complete} if every weak Cauchy sequence has a weak
limit. Because of the importance of this concept, we reiterate it.
\end{rmk}

\begin{dfn} (i) Let $X$ be a topological space and  $F\subseteq C(X)$. We say that $F$ has
the {\em weak sequential completeness property} (or short SCP) if
the limit of each pointwise convergent sequence $\{f_n\}\subseteq
F$ is continuous.

(ii) We say that a (bounded) family  $F$ of real-valued function
on a set $X$ has the {\em relative sequential compactness in
${\mathbb{R}}^X$} (short RSC) if   every sequence in $F$ has a
pointwise convergent subsequence in ${\mathbb{R}}^X$.
\end{dfn}

 The next result  is
another application of the Eberlein-Grothendieck criterion:

\begin{fct} \label{nip+scp=stable}
Let $X$ be a compact space and $A\subseteq C(X)$ be bounded. Then
$A$ is  relatively weakly compact in $C(X)$ iff it has RSC and
SCP.
\end{fct}
\begin{proof}
See Theorem 4.3 in \cite{K3}.
\end{proof}

\begin{pro} \label{SCP->NSOP}
If the set $\{\phi(x,a):a\in\mathcal{U}\}$ has the SCP, then
$\phi(x,y)$ is NSOP.
\end{pro}
\begin{proof} Suppose, for a contradiction, that  $\{\phi(x,a):a\in\mathcal{U}\}$ has the SCP
 and $\phi$ is SOP. By SOP, there are $(a_ib_i:i<\omega)$ in the monster model $\mathcal U$ and
$\epsilon>0$ such that $\phi({\mathcal
U},a_i)\leqslant\phi({\mathcal U},a_{i+1})$ and
$\phi(b_j,a_i)+\epsilon<\phi(b_i,a_j)$ for all $i<j$. Let $b$ be
a cluster point of $\{b_i\}_{i<\omega}$. By SCP,
$\phi(S_\phi({\mathcal U}),a_i)\nearrow\psi$ and $\psi$ is
continuous. But
$\lim_i\lim_j\phi(b_j,a_i)+\epsilon\leqslant\lim_i\lim_j\phi(b_i,a_j)$
and by continuity $\psi(b)+\epsilon\leqslant\psi(b)$, a
contradiction.
\end{proof}

\begin{cor}
 Suppose that $T$ is NIP and SCP. Then $T$ is stable.
\end{cor}
\begin{proof}
 Use  the Eberlein--\v{S}mulian theorem. (See also \ref{nip+scp=stable}  above.)
\end{proof}

\begin{fct}
Suppose that $T$ is a theory. Then the following are equivalent:
\begin{itemize}
             \item [{\em (i)}] $T$ is NSOP.
             \item [{\em (ii)}] For each indiscernible sequence $(a_n)_{n<\omega}$ and
formula $\phi(x,y)$, if the sequence $(\phi(x,a_n))_{n<\omega}$ is
\textbf{increasing}  on $S_\phi(\mathcal{U})$, then its limit is
continuous.
\end{itemize}
\end{fct}

\begin{proof} Immadiate by definition.
\end{proof}

\subsection{Shelah's theorem for continuous logic}
Now we want to give a proof of Shelah's theorem for continuous
logic. First we show that SCP and NwSOP are the same.  For this,
we need some definitions. Let $M$ be a saturated enough structure
and $\phi:M\times M\to{\Bbb R}$ a formula. For subsets
$B,D\subseteq M$, we say that $\phi(x,y)$ has the {\em order
property on } $B\times D$ (short OP on $B\times D$) if there are
$\epsilon>0$ and sequences $(a_i)\subseteq B$, $(b_i)\subseteq D$
such that $|\phi(a_i,b_j)-\phi(a_j,b_i)|\geqslant \epsilon$ for
all $i<j<\omega$. We will say that $\phi(x,y)$ has the {\em NIP on
$B\times D$}, if  for the set $A=\{\phi(a,y):S_y(D)\to{\Bbb
R}~|a\in B\}$, any of the cases in Lemma~3.12 in \cite{K3} holds.

\begin{pro} \label{SCP=NwSOP}
Suppose that $T$ is a theory. Then the following are equivalent:
\begin{itemize}
             \item [{\em (i)}] $T$ is NwSOP.
             \item [{\em (ii)}]  $T$ is SCP.
\end{itemize}
\end{pro}

\begin{proof} (ii)~$\Rightarrow$~(i) is by definition. For
(i)~$\Rightarrow$~(ii) we repeat the argument of Shelah's theorem
(see Proposition~4.6 of \cite{K3}).

Indeed,  suppose that $T$ is NOT SCP; this means that there are an
indiscernible sequence $(a_n)_{n<\omega}$ and a formula
$\phi(x,y)$ such that the sequence $(\phi(x,a_n))_{n<\omega}$
pointwise converges but its limit is not continuous. Since the
limit is not continuous, $\tilde{\phi}(y,x)=\phi(x,y)$ has OP on
$\{a_n\}_{n<\omega}\times S_\phi({\mathcal U})$. Since every
sequence in $\{\phi(x,a_n)\}_{n<\omega}$ has a pointwise
convergent subsequence, $\tilde{\phi}(y,x)$ is NIP on
$\{a_n\}_{n<\omega}\times S_\phi({\mathcal U})$.  The following
argument is classic (see   \cite{Poi} and \cite{S}).  Since
$\tilde{\phi}(y,x)$ has OP, there  are $r<s$ and a sequence
$\{b_N\}\subseteq S_\phi({\mathcal U})$ such that
$\tilde\phi(a_i,b_N)\leq r$ holds if $i<N$, and
$\tilde\phi(a_i,b_N)\geq s$ in the otherwise.
 By NIP, for each $r<s$ and $\epsilon\in(0,s-r)$,
there is some integer $n$ and $\eta : n \rightarrow \{0,1\}$ such
that $\bigwedge_{i<n} \tilde\phi(a_i,x)^{\eta(i)}$ is
inconsistent, where  for a formula $\varphi$, we use the notation
$\varphi^1$ to mean $\varphi\leq r+\frac{\epsilon}{2}$ and
$\varphi^0$ to mean $\varphi\geq s-\frac{\epsilon}{2}$. (Recall
that unlike classical model theory, in continuous logic Trus is 0
and False is 1.) Starting with that formula, we change one by one
instances of $\tilde\phi(a_i,x)\geq s-\frac{\epsilon}{2} \wedge
\tilde\phi(a_{i+1},x)\leq r+\frac{\epsilon}{2}$ to
$\tilde\phi(a_i,x)\leq r+\frac{\epsilon}{2} \wedge
\neg\tilde\phi(a_{i+1},x)\geq s-\frac{\epsilon}{2}$. Finally, we
arrive at a formula of the form $\bigwedge_{i<N}
\tilde\phi(a_i,x)\leq r+\frac{\epsilon}{2} \wedge
\bigwedge_{N\leq i<n}
 \tilde\phi(a_i,x)\geq s-\frac{\epsilon}{2}$. The tuple $b_N$ satisfies that formula.
 Therefore, for such $r<s$ and $\epsilon$, there is some
$i_0<n$, $\eta_0 : n \rightarrow \{0,1\}$ such that
$$\bigwedge_{i\neq i_0, i_0+1} \tilde\phi(a_i,x)^{\eta_0(i)} \wedge \tilde\phi(a_{i_0},x) \geq
s-\frac{\epsilon}{2}\wedge \tilde\phi(a_{i_0+1},x)\leq
r+\frac{\epsilon}{2}$$ is inconsistent, but
$$\bigwedge_{i\neq i_0, i_0+1} \tilde\phi(a_i,x)^{\eta_0(i)} \wedge \tilde\phi(a_{i_0},x)\leq r+\frac{\epsilon}{2} \wedge \tilde\phi(a_{i_0+1},x)\geq
s-\frac{\epsilon}{2}$$ is consistent. Let us define $\varphi(\bar
a,x)=\bigwedge_{i\neq i_0,i_0+1} \tilde\phi(a_i,x)^{\eta_0(i)}$.
Increase the sequence $(a_i : i<\omega)$ to an indiscernible
sequence $(a_i:i\in \mathbb Q)$. Then for $i_0 \leq i<i' \leq
i_0+1$, the formula $\varphi(\bar a,x)\wedge \tilde\phi(a_i,x)\leq
r+\frac{\epsilon}{2}  \wedge \tilde\phi(a_{i'},x)\geq
s-\frac{\epsilon}{2}$ is consistent, but $\varphi(\bar a,x) \wedge
\tilde\phi(a_i,x)\geq s-\frac{\epsilon}{2} \wedge
\tilde\phi(a_{i'},x)\leq r+\frac{\epsilon}{2}$ is inconsistent.
Thus the formula $\psi(x,y) = \varphi(\bar a,x) \wedge
\tilde\phi(y,x)$ is the formula $\psi_n$ (for some $n$) in the
definition of wSOP above. Note that for all $b$, the sequence
$\psi(b,a_i)$ has eventual true value; equivalently it converges.
(Indeed, since the sequence $(\tilde{\phi}(a_i,x):i<\omega)$
converges and we increased the sequence $(a_i : i<\omega)$ to the
indiscernible sequence $(a_i:i\in \mathbb Q)$, it is easy to
verify that every sequence
$(\tilde{\phi}(a_{j_i},x):i_0<j_i<j_{i+1}<i_0+1,~ i<\omega)$
converse. Assume not, and for some $b$ the sequence
$\tilde{\phi}(a_{j_i},b)$ diverges. Take the strictly increasing
function $\tau:\omega\to\omega$ by $\tau(j_i)=i$. By
indiscernibility, the set of conditions
$\{\tilde{\phi}(a_i,x)=\tilde{\phi}(a_{j_i},b):i<\omega\}$ is
consistence; but this means that for some $b$ the sequence
$\tilde{\phi}(a_i,b)$ diverges, a contradiction.)
As $\epsilon$ is arbitrary, the proof is completed.
\end{proof}

The next result  is a generalization of Shelah's theorem
(\cite{Sh}, Theorem~4.1) for continuous logic.

\begin{cor}[Shelah's theorem for continuous logic] \label{Shelah-continuous}  Suppose that $T$ is NIP and NwSOP. Then $T$
is stable.
\end{cor}
\begin{proof}  Let $\phi(x,y)$ be a formula, $(a_n)_{n<\omega}$ an
indiscernible sequence, and $(b_n)_{n<\omega}$ an arbitrary
sequence. Suppose that the double limits
$\lim_m\lim_n\phi(b_n,a_m)$ and $\lim_n\lim_m\phi(b_n,a_m)$ exist.
By NIP, there is a convergent subsequence $\phi(x,a_{m_k})$ such
that $\phi(x,a_{m_k})\to\psi(x)$ on $S_\phi(\mathcal{U})$.
Therefore, $\lim_n\lim_k\phi(b_n,a_{m_k})=\lim_n\psi(b_n)$ and
$\lim_k\lim_n\phi(b_n,a_{m_k})=\lim_k\phi(b,a_{m_k})=\psi(b)$
where $b$ is a cluster point of $\{b_n\}$. By NwSOP (or
equivalently SCP), $\lim_n\psi(b_n)=\psi(b)$. So the double limits
are the same and thus $T$ is stable. (Compare
Fact~\ref{nip+scp=stable}.)
\end{proof}

\subsection{Universal models of Banach lattices} We show that a
formuls in the language of Banach lattices has SOP$_n$ and so for
many of cardinals there is not any universal model. 

In \cite{SU}, Shelah and Usvyatsov proved that the theory $T_B$
of all Banach spaces is quantifier-free-NSOP, i.e. there is not a
\textbf{quantifier-free} formula such that defines a partial
order with infinite chain. Also, they showed that a
\textbf{quantifier-free} formula has SOP$_4$ (even SOP$_n$ for
$n\geq 4$). Using the Shelah's result, this implies that for many
cardinals there is not a universal model of Banach spaces. Note
that since such the formula is \textbf{quantifier-free}, every
\textbf{subspace} is an embedding, so universal model does not
exist in the sense of Banach theorists. Of course, $T_B$ has SOP
using a formula with a quantifier. Indeed, consider the formula
$\phi(x,y)=\max(\|x+y\|,\|x-y\|)$. Let $s_n=e_1+\ldots+e_n$ where
$(e_n)$ is the standard basis of $c_0$. Now
$\phi(e_k,s_n)+\frac{1}{2}<\phi(e_k,e_m)$ for all $n<k\leq
m<\omega$ AND $\phi(x,e_n)\leq\phi(x,e_m)$ for all $x\in c_0$ and
$n\leq m$. Let $\psi(x_1,x_2):=\forall
x(\phi(x,x_1)\to\phi(x,x_2))$. Then $\psi(x,y)$ define a partial
order with an infinite chain in the monster model of Banach
spaces. (Recall that an \textbf{incomplete} theory has SOP if a
complete extension of it has SOP. In this case,  the
Kojman--Shelah result holds still.)

On the other hand, in \cite{FHV} it is showen that the class of
C$^*$-algebras has SOP with a quantifier-free formula. (Note that
its theory is incomplete.) So, using Kojma--Shelah, this implies
non-existence of universal models in many cardinals. Here we want
to show that the calls of Banach lattice has SOP$_4$ with a
quantifier-free formula. Indeed let $\phi(x,y)=\||x|+|y|\|$. Then
$\phi(e_k,s_n)+\frac{1}{2}<\phi(e_k,e_m)$ for all $n<k\leq
m<\omega$ AND $\phi(x,e_n)\leq\phi(x,e_m)$ for all $x\in c_0$ and
$n\leq m$ (where $\phi(e_k,s_n)+\frac{1}{2}<\phi(e_k,e_m)$ for
all $n<k\leq m<\omega$ AND $\phi(x,e_n)\leq\phi(x,e_m)$ for all
$x\in c_0$ and $n\leq m$). Let
$\psi(x_1x_2,y_1y_2):=(\|x_2+y_1\|\leq 1 \wedge \|x_1+y_2\|\geq
2)$. Now $\psi(e_is_i,e_{i+1}s_{i+1})$ holds for all $i<\omega$.
This means that there is an infinite chain. It is easy to check
that $\psi$ has SOP$_4$, using the triangle property of norm. 
Note that $\psi$ is \textbf{quantifier-free}.

Since the above formula $\psi$ is \textbf{quantifier-free} we
have:

\begin{cor} Suppose there exists a universal Banach lattice (under
isometry) in $\lambda=cf(\lambda)$. Then either
$\lambda=\lambda^{<\lambda}$ or $\lambda=\mu^+$ and $2^{<\mu}\leq
\lambda$.
\end{cor}

 {\bf Acknowledgements.}
 I want to thank Alexander Usvyatsov for his comments and John T. Baldwin for his interest
in reading of a preliminary version of this article and for his
comments.


\vspace{10mm}


\end{document}